\documentclass[11pt]{amsart}
\usepackage{enumerate,amssymb,amsfonts,amsmath}
\usepackage{hyperref}
\usepackage{xcolor}

\newtheorem{theorem}{Theorem}
\newtheorem*{example}{Example}
\newtheorem{lemma}{Lemma}
\newtheorem{cl}{Claim}

\newcommand{\w}{\widetilde}

\usepackage{vmargin}
\setpapersize{A4}
\setmargins{2.5cm}      
{1.5cm}                        
{16.5cm}                      
{23.42cm}                   
{10pt}                          
{1cm}                           
{0pt}                          
{2cm}

\begin{document}

\title[On 2-local diameter-preserving maps between $C(X)$-spaces]{On 2-local diameter-preserving maps between $C(X)$-spaces}

\author{A. Jim{\'e}nez-Vargas}
\address{Departamento de Matem{\'a}ticas, Universidad de Almer{\'i}a, 04120, Almer{\i}a, Spain}
\email{ajimenez@ual.es}

\author{Fereshteh Sady}
\address{Department of Pure Mathematics, Faculty of Mathematical Sciences, Tarbiat Modares University, Tehran 14115-134, Iran}
\email{sady@modares.ac.ir}

\date{\today}


\begin{abstract}
The 2-locality problem of diameter-preserving maps between $C(X)$-spaces is addressed in this paper. For any compact Hausdorff space $X$ with at least three points, we give an example of a 2-local diameter-preserving map on $C(X)$ which is not linear. However, we show that for first countable compact Hausdorff spaces $X$ and $Y$, 
every 2-local diameter-preserving map from $C(X)$ to $C(Y)$ is linear and surjective up to constants in some sense. This yields the 2-algebraic reflexivity of isometries with respect to the diameter norms on the quotient spaces. 
\end{abstract}
\maketitle

\section{Introduction and results}

Let $E$ and $F$ be Banach spaces and let $\mathcal{S}$ be a subset of $\mathcal{L}(E,F)$, the space of linear operators from $E$ to $F$. Let us recall that a linear map $T\colon E\to F$ is a local $\mathcal{S}$-map if for every $e\in E$, there exists a $T_e\in\mathcal{S}$, depending possibly on $e$, such that $T_e(e)=T(e)$. On the other hand, a map $\Delta\colon E\to F$ (which is not assumed to be linear) is called a 2-local $\mathcal{S}$-map if for any $e,u\in E$, there exists a $T_{e,u}\in\mathcal{S}$, depending in general on $e$ and $u$, such that $T_{e,u}(e)=\Delta(e)$ and $T_{e,u}(u)=\Delta(u)$. 

Most of the published works on local and 2-local $\mathcal{S}$-maps concern the set $\mathcal{S}=\mathcal{G}(E)$, the group of surjective linear isometries of $E$. In this case, the local and 2-local $\mathcal{G}(E)$-maps are known as local and 2-local isometries of $E$, respectively. The main question which one raises is for which Banach spaces, every local isometry is a surjective isometry or, equivalently, which Banach spaces have an algebraically reflexive isometry group. In the 2-local setting, the basic problem is to show that every 2-local isometry is a surjective linear isometry. 

In \cite{Mol-02}, Moln\'{a}r initiated the study of 2-local isometries on operator algebras and proposed to research the 2-locality of isometries on function algebras. In this line, Gy\H{o}ry \cite{Gyo-01} dealed with 2-local isometries on spaces of continuous functions. In \cite{JimVil-11}, Villegas and the first author adapted the Gy\H{o}ry's technique to analyze the 2-local isometries on Lipschitz algebras. Hatori, Miura, Oka and Takagi \cite{HatMiuOkaTak-07} considered 2-local isometries on uniform algebras including certain algebras of holomorphic functions. More recently, Hosseini \cite{Hos-17}, Hatori and Oi \cite{HatOi-18c} and Li, Peralta, L. Wang and Y.-S. Wang \cite{LiPerWanWan-19} have investigated 2-local isometries of different function algebras such as uniform algebras, Lipschitz algebras and algebras of continuously differentiable functions.

Our aim in this paper is to study the 2-locality problem for isometries between certain quotient Banach spaces which appear in a natural form when one treats with maps between $C(X)$-spaces which preserve the diameter of the range.

Let $C(X)$ be the Banach space of all continuous complex-valued functions on a compact Hausdorff space $X$, with the usual supremum norm. A map $\Delta\colon C(X)\to C(Y)$ (not necessarily linear) is diameter-preserving if 
$$
\rho(\Delta(f)-\Delta(g))=\rho(f-g)\qquad (f,g\in C(X)),
$$
where for each $f\in C(X)$,  
$$
\rho(f)=\sup\left\{\left|f(x)-f(z)\right|\colon x,z\in X\right\}.
$$
Gy\H{o}ry and Moln\'{a}r \cite{GyoMol-98} introduced this kind of maps and gave a complete description of diameter-preserving linear bijections of $C(X)$, when $X$ is a first countable compact Hausdorff space. Cabello S\'anchez \cite{Cab-99} and Gonz\'{a}lez and Uspenskij \cite{GonUsp-99} established the same characterization without the first countability assumption. As usual, $\mathbb{T}$ denotes the unit circle of $\mathbb{C}$. We also put 
$$
\mathbb{T}^+=\{e^{it}\colon t\in [0,\pi)\}.  
$$
Moreover, $1_X$ and $0_X$ stand for the constant functions $1$ and $0$ on $X$, respectively. 

\begin{theorem}\label{IsoLipWeaver}\cite{Cab-99, GonUsp-99, GyoMol-98}. 
Let $X$ and $Y$ be compact Hausdorff spaces. A linear bijection $T\colon C(X)\to C(Y)$ is diameter-preserving if and only if there exist a homeomorphism $\phi\colon Y\to X$, a linear functional $\mu\colon C(X)\to\mathbb{C}$ and a number $\lambda\in\mathbb{T}$ with $\lambda\neq-\mu(1_X)$ such that 
$$
T(f)=\lambda f\circ\phi+\mu(f)1_Y\qquad \left(f\in C(X)\right).
$$
\end{theorem}

The main problem addressed in the study of diameter-preserving maps between function algebras is establishing a representation of such maps as the sum of a weighted composition operator and a functional as in Theorem \ref{IsoLipWeaver}. We have a precise description of diameter-preserving maps for most of the classical function spaces (see for example \cite{AizRam-07, AizRam-10, AizTam-07, Cab-99, FonHos-17, FonSan-04, RaoRoy-01} for diameter-preserving linear maps and \cite{BarRoy-02, FonHos-19, JamSad-16} for the non-linear case). 

In the case in which $\mathcal{S}$ is the set of all diameter-preserving linear bijections from $C(X)$ to $C(Y)$, we studied in a recent paper \cite{JimSad-20} the local $\mathcal{S}$-maps, referred there to as local diameter-preserving maps. Namely, we proved that in the case where $X$ and $Y$ are first countable, every local diameter-preserving map from $C(X)$ to $C(Y)$ is a diameter-preserving bijection. The first countability on the topological spaces is a mild and appropriate condition when one addresses these problems. For example, 
the isometry group and the automorphism group of $C(X)$ are algebraically reflexive in case $X$ is first countable \cite{MolZal-99}, but that reflexivity fails if $X$ is not (see section 7 in \cite{CabMol-02}).

It is natural to arise the corresponding question in the 2-local context, that is, is every 2-local diameter-preserving map a diameter-preserving linear bijection? Unfortunately or not, the answer is negative as we see next with a counterexample. 

Let us recall that a map $\Delta\colon C(X)\to C(Y)$ (not assumed to be linear) is a 2-local diameter-preserving map if for any $f,g\in C(X)$, there exists a diameter-preserving linear bijection $T_{f,g}$ from $C(X)$ to $C(Y)$ such that $T_{f,g}(f)=\Delta(f)$ and $T_{f,g}(g)=\Delta(g)$.

\begin{example}{\rm 
\textbf{(A 2-local diameter-preserving non-linear map between $C(X)$-spaces).}  
Let $X$ and $Y$ be homeomorphic compact Hausdorff spaces with at least three points. Let $\phi\colon Y\to X$ be a homeomorphism and let $\mu\colon C(X)\to\mathbb{C}$ be a homogeneous non-additive functional such that 
$\mu(1_X)\neq -1$ and $\mu(1_X-f)=\mu(1_X)-\mu(f)$ for all $f\in C(X)$. To give an example of such a functional $\mu$, fix three distinct points $x_1,x_2,x_3\in X$ and define $\mu\colon C(X)\to\mathbb{C}$ by 
$$
\mu(f)=\left\{\begin{array}{lll}
f(x_1)& &\text{ if } f(x_1)=f(x_2) \text{ and } f(x_1)\neq f(x_3),\\
      & &\\
f(x_3)& &\text{otherwise.} 
\end{array}\right.
$$
It is easy to see that $\mu$ is homogeneous and $\mu(1_X-f)=\mu(1_X)-\mu(f)$ for all $f\in C(X)$. Meanwhile, $\mu$ is not additive, since we can take $f,g\in C(X)$ such that $f(x_1)=f(x_2)=1$ and $f(x_3)=0$ and also $g(x_1)=g(x_3)=1$ and $g(x_2)=0$, and then $\mu(f+g)=1\neq 2=\mu(f)+\mu(g)$. 

Define now the map $\Delta\colon C(X)\to C(Y)$ by 
$$
\Delta(f)=f\circ\phi+\mu(f)1_Y \qquad (f\in C(X)).
$$
For each pair $f,g\in C(X)$, consider a linear functional $\mu_{f,g}\colon C(X)\to \mathbb{C}$ satisfying
$$
\mu_{f,g}(f)=\mu(f),\quad \mu_{f,g}(g)=\mu(g),\quad\mu_{f,g}(1_X)=\mu(1_X).
$$ 
Notice that such a functional $\mu_{f,g}$ exists. Indeed, if $\{f,g,1_X\}$ is linearly independent, the existence of $\mu_{f,g}$ can be established by extending linearly to $C(X)$ a convenient linear functional defined on ${\rm span}\{f,g,1_X\}$. 
If $\{f,g,1_X\}$ is linearly dependent and $1_X\in {\rm span}\{f,g\}$, then we can find a linear functional $\mu_{f,g}$ on $C(X)$ such that $\mu_{f,g}(f)=\mu(f)$ and $\mu_{f,g}(g)=\mu(g)$ (note that $\mu$ is homogeneous). Since $1_X=af+bg$ for some $a,b\in\mathbb{C}$,  the hypotheses on $\mu$ easily imply that $\mu_{f,g}(1_X)=\mu(1_X)$, as desired. In the case where $\{f,g,1_X\}$ is linearly dependent and $1_X\notin {\rm span}\{f,g\}$ we conclude that $f$ and $g$ are linearly dependent and we may assume that $f=cg$ for some scalar $c$. In this case, there exists a linear functional $\mu_{f,g}$ on $C(X)$ such that $\mu_{f,g}(1_X)=\mu(1_X)$ and $\mu_{f,g}(g)=\mu(g)$. Hence $\mu_{f,g}(f)=\mu(f)$ since $\mu$ is homogeneous. Thus in each case we can find a linear functional $\mu_{f,g}\colon C(X)\to \mathbb{C}$ with the desired properties. 
 
Finally, for any $f,g\in C(X)$, define $T_{f,g}\colon C(X)\to C(Y)$ by 
$$
T_{f,g}(h)=h\circ\phi+ \mu_{f,g}(h)1_Y \qquad (h\in C(X)).
$$
Then $T_{f,g}$ is a diameter-preserving linear bijection by Theorem \ref{IsoLipWeaver}. Clearly, for any $f,g\in C(X)$, we have $T_{f,g}(f)=\Delta(f)$ and $T_{f,g}(g)=\Delta(g)$. Hence $\Delta$ is a 2-local diameter-preserving map which is homogeneous but not additive.  
}
\end{example}
  
However, we shall show here that, in the case where $X$ and $Y$ are first countable, every 2-local diameter-preserving map (which is immediately diameter-preserving) is linear and surjective up to constants in some sense. Our approach consists in analysing the 2-local isometries of the following quotient Banach spaces which appear closely related to diameter-preserving maps.  

Given a compact Hausdorff space $X$, let $C_\rho(X)$ denote the quotient space $C(X)/\ker(\rho)$. Clearly, $C_\rho(X)$ is a Banach space with the norm
$$
\left\|\pi_X(f)\right\|_\rho=\rho(f)\qquad (f\in C(X)),
$$
where $\pi_X\colon C(X)\to C_\rho(X)$ is the canonical quotient surjection. Let us recall that a mapping $T\colon  C_\rho(X)\to C_\rho(Y)$ (not taken linear nor surjective) is an isometry whenever
$$
\left\|T(\pi_X(f))-T(\pi_X(g))\right\|_\rho=\left\|\pi_X(f)-\pi_X(g)\right\|_\rho\qquad (f,g\in C(X)).
$$

Our main result is the following theorem on 2-local isometries between $C_\rho(X)$-spaces.

\begin{theorem}\label{main2}
Let $X$ and $Y$ be first countable compact Hausdorff spaces and let $T\colon C_\rho(X)\to C_\rho(Y)$ be a 2-local isometry. Then $T$ is a surjective linear isometry. 
\end{theorem}

\section{Proofs}

The first key tool to prove Theorem \ref{main2} is the fact that every isometry $T$ between $C_\rho(X)$-spaces induces a convenient (injective) diameter-preserving map $\Delta$ between the corresponding $C(X)$-spaces which is linear or surjective if so is $T$. Towards this end, fix two points $u_0\in X$ and $w_0\in Y$ and consider the linear bijections 
$$
\Psi_X\colon C(X)\to C_\rho(X)\oplus\mathbb{C}, \quad \Psi_X(f)=(\pi_X(f),f(u_0))\qquad (f\in C(X))
$$
and 
$$
\Psi_Y\colon C(Y)\to C_\rho(Y)\oplus\mathbb{C}, \quad \Psi_Y(g)=(\pi_Y(g),g(w_0))\qquad (g\in C(Y)).
$$

\begin{lemma}\label{lem1}
Let $X$ and $Y$ be compact Hausdorff spaces and let $T\colon C_\rho(X)\to C_\rho(Y)$ be an isometry. Then $\Delta\colon C(X)\to C(Y)$ defined by 
$$
\Delta(f)=\Psi_Y^{-1}(T(\pi_X(f)),f(u_0)) \qquad (f\in C(X))
$$
is an injective diameter-preserving map. Moreover, $T$ is linear (respectively, surjective) if and only if so is $\Delta$. 
\end{lemma}

\begin{proof}
Given $f,g\in C(X)$, we put $h=\Delta(f)-\Delta(g)$. Then 
\begin{align*}
h &=\Psi_Y^{-1}(T(\pi_X(f)),f(u_0))-\Psi_Y^{-1}(T(\pi_X(g)),g(u_0))\\
  &=\Psi_Y^{-1}(T(\pi_X(f))-T(\pi_X(g)),f(u_0)-g(u_0)).
\end{align*}
Hence  
$$
(\pi_Y(h),h(w_0))=\Psi_Y(h)=(T(\pi_X(f))-T(\pi_X(g)),f(u_0)-g(u_0)),
$$
and consequently $\pi_Y(h)=T(\pi_X(f))-T(\pi_X(g))$. This implies that 
\begin{align*}
\rho(\Delta(f)-\Delta(g))&=\left\|\pi_Y(h)\right\|_\rho=\left\|T(\pi_X(f))-T(\pi_X(g))\right\|_\rho\\
                         &=\left\|\pi_X(f)-\pi_X(g)\right\|_\rho\\
                         &=\rho(f-g),
\end{align*}
that is, $\Delta$ is diameter-preserving. Clearly, $\Delta$ is injective. It is also easy to see that $\Delta$ is linear if so is $T$. Assume now that $T$ is surjective. Then, given $g\in C(Y)$ there exists $f\in C(X)$ such that $T(\pi_X(f))=\pi_Y(g)$. Replacing $g$ by $g+\lambda$ for some $\lambda\in\mathbb{C}$, we can assume that $g(w_0)=f(u_0)$. Hence 
$$
\Delta(f)=\Psi_Y^{-1}(T(\pi_X(f)),f(u_0))=\Psi_Y^{-1}(\pi_Y(g),g(w_0))=g,
$$
which shows that $\Delta$ is surjective, as well. A similar reasoning justifies that if $\Delta$ is linear (respectively, surjective), then so is $T$. 
\end{proof}

Now, we prove our main theorem.

\begin{proof} \textit{(Theorem \ref{main2}).}  
Let $T\colon C_\rho(X)\to C_\rho(Y)$ be a 2-local isometry. The proof will be carried out through a series of claims. The proofs of some of them are similar to those of the corresponding steps in the proof of the main theorem (Theorem 2) of \cite{JimSad-20}. For this reason we shall only include here the proof of those claims whose arguments differ essentially from similar steps  in \cite{JimSad-20}.

\begin{cl}\label{cl1}
The map $\Delta\colon C(X)\to C(Y)$ defined by 
$$
\Delta(f)=\Psi_Y^{-1}(T(\pi_X(f)),f(u_0)) \qquad (f\in C(X))
$$
is a 2-local diameter-preserving map.
\end{cl}

Let $f,g\in C(X)$. By hypotheses, there exists a surjective linear isometry $T_{f,g}\colon C_\rho(X)\to C_\rho(Y)$ such that $T_{f,g}(\pi_X(f))=T(\pi_X(f))$ and $T_{f,g}(\pi_X(g))=T(\pi_X(g))$. Define $\Delta_{f,g}\colon C(X)\to C(Y)$ by 
$$
\Delta_{f,g}(h)=\Psi_Y^{-1}(T_{f,g}(\pi_X(h)),h(u_0))\qquad (h\in C(X)).
$$
By Lemma \ref{lem1}, $\Delta_{f,g}$ is a diameter-preserving linear bijection from $C(X)$ to $C(Y)$ satisfying  $\Delta_{f,g}(f)=\Delta(f)$ and $\Delta_{f,g}(g)=\Delta(g)$.

\medskip

The following fact will be used repeatedly without any explicit mention in our proof. 

\begin{cl}\label{cl3}
For any $f,g\in C(X)$, there exists a diameter-preserving linear bijection $\Delta_{f,g}$ of $C(X)$ to $C(Y)$ such that $\Delta_{f,g}(f)=\Delta(f)$ and $\Delta_{f,g}(g)=\Delta(g)$. Moreover, there exist a homeomorphism $\phi_{f,g}\colon Y\to X$, a linear functional $\mu_{f,g}$ on $C(X)$ and a number $\lambda_{f,g}\in\mathbb{T}$ with $\lambda_{f,g}\neq -\mu_{f,g}(1_X)$ such that 
$$
\Delta(f)(y)=\lambda_{f,g} f(\phi_{f,g}(y))+\mu_{f,g}(f)\qquad (y\in Y)
$$
and 
$$
\Delta(g)(y)=\lambda_{f,g} g(\phi_{f,g}(y))+\mu_{f,g}(g)\qquad (y\in Y).
$$
\end{cl}

It follows from Claim \ref{cl1} and Theorem \ref{IsoLipWeaver}. 

\begin{cl}\label{cl2}
$\Delta$ is injective, diameter-preserving and homogeneous.
\end{cl}

Let $f,g\in C(X)$. If $\Delta(f)=\Delta(g)$, then $f=g$ by the injectivity of $\Delta_{f,g}$ and therefore $\Delta$ is injective. Clearly, $\Delta$ is diameter-preserving because 
$$
\rho(\Delta(f)-\Delta(g))=\rho(\Delta_{f,g}(f)-\Delta_{f,g}(g))=\rho(f-g).
$$
Finally, given $\lambda\in\mathbb{C}$, we have 
$$
\Delta(\lambda f)=\Delta_{f,\lambda f}(\lambda f)=\lambda \Delta_{f,\lambda f}(f)=\lambda\Delta(f),
$$
and thus $\Delta$ is homogeneous. 

\medskip

By Claim \ref{cl3}, there exists a homeomorphism from $Y$ onto $X$. Hence $Y$ and $X$ have the same cardinality. Since Theorem \ref{main2} is quite easy to verify when $Y$ is a singleton, we suppose from now on that $X$ and $Y$ have at least two points.

\medskip

Given a set $X$ with cardinal number $|X|\geq 2$, we set  
\begin{align*}
\widetilde{X}&=\left\{(x_1,x_2)\in X\times X\colon x_1\neq x_2\right\},\\
X_2&=\left\{\{x_1,x_2\}\colon (x_1,x_2)\in\widetilde{X}\right\}, 
\end{align*}
and we define the natural correspondence  $\Lambda_X \colon  \w{X} \to X_2$ by 
$$
\Lambda_X\left( (x_1,x_2)\right)=\{x_1,x_2\} \qquad \left( (x_1,x_2)\in \w{X}\right).
$$

Given a compact Hausdorff space $X$ and a point $(x_1,x_2)\in\widetilde{X}$, the Urysohn's lemma guarantees the existence of a continuous function $h_{(x_1,x_2)}\colon X\to [0,1]$ such that 
$$
h_{(x_1,x_2)}(x_1)-h_{(x_1,x_2)}(x_2)=\rho(h_{(x_1,x_2)}).
$$
In fact, $h_{(x_1,x_2)}(x_1)=1$ and $h_{(x_1,x_2)}(x_2)=0$. Furthermore, since  $X$ is also first countable, we can take $h_{(x_1,x_2)}$ such that $h_{(x_1,x_2)}^{-1}(\{1\})=\{x_1\}$ and $h_{(x_1,x_2)}^{-1}(\{0\})=\{x_2\}$. In particular,  
$$
\left|h_{(x_1,x_2)}(z)-h_{(x_1,x_2)}(w)\right|<\rho(h_{(x_1,x_2)})
$$
for all $(z,w)\in\widetilde{X}\setminus\{(x_1,x_2),(x_2,x_1)\}$. 

\begin{cl}\label{cl4}
For any $(x_1,x_2)\in\widetilde{X}$, the set 
$$
\mathcal{B}_{(x_1,x_2)}=\bigcap_{f\in C(X)}\mathcal{B}_{(x_1,x_2),f}
$$
is nonempty, where 
$$ 
\mathcal{B}_{(x_1,x_2),f}
=\left\{((y_1,y_2),\lambda)\in\widetilde{Y}\times \mathbb{T}\colon\Delta(f)(y_1)-\Delta(f)(y_2)=\lambda\left(f(x_1)-f(x_2)\right)\right\}\qquad (f\in C(X)).
$$
\end{cl}

Let $(x_1,x_2)\in\widetilde{X}$. Given $f\in C(X)$, the set $\mathcal{B}_{(x_1,x_2),f}$ is a nonempty subset of $\widetilde{Y}\times \mathbb{T}$. Indeed,  it suffices to choose $y_1,y_2\in Y$ such that $\phi_{f,f}(y_i)=x_i$ for $i=1,2$. Then 
$$
\Delta(f)(y_1)-\Delta(f)(y_2)=\Delta_{f,f}(f)(y_1)-\Delta_{f,f}(f)(y_2)= \lambda_{f,f}\left(f(x_1)-f(x_2)\right),
$$
that is, $((y_1,y_2),\lambda_{f,f})\in \mathcal{B}_{(x_1,x_2),f}$. An easy verification shows that $\mathcal{B}_{(x_1,x_2),f}$ is also closed in $\widetilde{Y}\times \mathbb{T}$.  



We next prove that the family $\left\{\mathcal{B}_{(x_1,x_2),f}\colon f\in C(X)\right\}$ has the finite intersection property. Let $n\in\mathbb{N}$ and $f_1,\ldots,f_n\in C(X)$. Take the function $g=h_{(x_1,x_2)}$. For each $i\in\{1,\ldots,n\}$, there exists a diameter-preserving linear bijection $\Delta_{g,f_i}$ from $C(X)$ to $C(Y)$ such that $\Delta_{g,f_i}(g)=\Delta(g)$ and $\Delta_{g,f_i}(f_i)=\Delta(f_i)$. Furthermore, we have a homeomorphism $\phi_{g,f_i}$ from $Y$ onto $X$, a linear functional $\mu_{g,f_i}$ on $C(X)$ and a number $\lambda_{g,f_i}\in\mathbb{T}$ with $\lambda_{g,f_i}\neq -\mu_{g,f_i}(1_X)$ such that 
$$
\Delta_{g,f_i}(h)(y)=\lambda_{g,f_i}h(\phi_{g,f_i}(y))+\mu_{g,f_i}(h)\qquad (h\in C(X),\; y\in Y).
$$
Let $((y_1,y_2),\lambda)\in\mathcal{B}_{(x_1,x_2),g}$ be arbitrary. For each $i\in\{1,\ldots,n\}$, we obtain 
\begin{align*}
\lambda\left(g(x_1)-g(x_2)\right)
&=\Delta(g)(y_1)-\Delta(g)(y_2)\\
&=\Delta_{g,f_i}(g)(y_1)-\Delta_{g,f_i}(g)(y_2)\\
&=\lambda_{g,f_i}\left(g(\phi_{g,f_i}(y_1))-g(\phi_{g,f_i}(y_2))\right)
\end{align*}
and therefore
$$
\left|g(\phi_{g,f_i}(y_1))-g(\phi_{g,f_i}(y_2))\right|=1.
$$
This implies that either 
$$
(\phi_{g,f_i}(y_1),\phi_{g,f_i}(y_2))=(x_1,x_2),
$$
or 
$$
(\phi_{g,f_i}(y_1),\phi_{g,f_i}(y_2))=(x_2,x_1).
$$ 
Hence $\lambda_{g,f_i}=\lambda$ in the first case, or $\lambda_{g,f_i}=-\lambda$ in the second one. We deduce that $\mathcal{B}_{(x_1,x_2),g}$ is contained in the set
$$ 
\left\{((\phi_{g,f_i}^{-1}(x_1),\phi_{g,f_i}^{-1}(x_2)),\lambda_{g,f_i}),
((\phi_{g,f_i}^{-1}(x_2),\phi_{g,f_i}^{-1}(x_1)),-\lambda_{g,f_i})\right\}.
$$
Now, for any $i\in\{1,\ldots,n\}$, we have 
\begin{align*}
\Delta(f_i)(y_1)-\Delta(f_i)(y_2)&=\Delta_{g,f_i}(f_i)(y_1)-\Delta_{g,f_i}(f_i)(y_2)\\
                                 &=\lambda_{g,f_i}\left(f_i(\phi_{g,f_i}(y_1))-f_i(\phi_{g,f_i}(y_2))\right)\\
                                 &=\lambda\left(f_i(x_1)-f_i(x_2)\right),
\end{align*}
whence $((y_1,y_2),\lambda)\in\mathcal{B}_{(x_1,x_2),f_i}$ and thus 
$$
\emptyset\neq \mathcal{B}_{(x_1,x_2),g}\subseteq \bigcap_{i=1}^n\mathcal{B}_{(x_1,x_2),f_i}.
$$
This proves that $\left\{\mathcal{B}_{(x_1,x_2),f}\colon f\in C(X)\right\}$ has the finite intersection property, and since $\mathcal{B}_{(x_1,x_2),g}$ is a compact subset of $\widetilde{Y}\times \mathbb{T}$, then $\mathcal{B}_{(x_1,x_2)}$ will be nonempty.  

\medskip

The proof of Claim \ref{cl5} is similar to Step 4 of  \cite{JimSad-20}.  

\begin{cl}\label{cl5} 
For every $(x_1,x_2)\in\widetilde{X}$, there exist $(y_1,y_2)\in\widetilde{Y}$ and $\lambda\in\mathbb{T}$ such that 
$$
\mathcal{B}_{(x_1,x_2)}=\left\{((y_1,y_2),\lambda),((y_2,y_1),-\lambda)\right\}.
$$
\end{cl}

It is immediate from Claim \ref{cl5} that for every $(x_1,x_2)\in \w{X}$, the set 
$$
\mathcal{A}_{(x_1,x_2)}=\left\{(y_1,y_2)\in \w{Y} \,|\,  \exists \lambda\in \mathbb{T}^+ \colon ((y_1,y_2),\lambda)\in\mathcal{B}_{(x_1,x_2)}\right\}
$$
is a singleton. Let $\Gamma\colon \w{X}\to \w{Y}$ be the map given by $\Gamma((x_1,x_2))=(y_1,y_2)$ where for each  $(x_1,x_2)\in \w{X}$, the element $(y_1,y_2)\in \w{Y}$ is the unique point of $A_{(x_1,x_2)}$. We note that if  $\mathcal{A}_{(x_1,x_2)}=\{(y_1,y_2)\}$, then the definition of $\mathcal{A}_{(x_1,x_2)}$ shows that there exists a (unique) scalar $\beta(x_1,x_2)\in \mathbb{T}^+$, depending on the pair $(x_1,x_2)$,  such that 
$$
\Delta(f)(y_1)-\Delta(f)(y_2)=\beta(x_1,x_2) \left(f(x_1)-f(x_2)\right) \qquad (f\in C(X)). 
$$ 
This concludes that 
$$
\Delta(f)(y_2)-\Delta(f)(y_1)=\beta(x_1,x_2) \left(f(x_2)-f(x_1)\right) \qquad (f\in C(X)),
$$
that is $\beta(x_2,x_1)=\beta(x_1,x_2)$ and  $\Gamma\left( (x_2,x_1)\right)=(y_2,y_1)$.     
\begin{cl}\label{cl8}
The map $\Gamma$ is a bijection from $\w{X}$ to $\cup_{(x_1,x_2)\in \w{X}}\mathcal{A}_{(x_1,x_2)}$.
\end{cl}
The surjectivity of $\Gamma$ is immediate, since $(y_1,y_2)=\Gamma\left((x_1,x_2)\right)$ if and only if $(y_1,y_2)\in \mathcal{A}_{(x_1,x_2)}$. 
To prove its injectivity, let $(x_1,x_2),(x_3,x_4)\in \w{X}$ be such that
$$
(y_1,y_2)=\Gamma\left( (x_1,x_2) \right)=\Gamma\left( (x_3,x_4)\right).
$$
Then  we have 
$$
\beta(x_1,x_2)\left(f(x_1)-f(x_2)\right)=\Delta(f)(y_1)-\Delta(f)(y_2)=\beta(x_3,x_4)\left(f(x_3)-f(x_4)\right)                                             
$$
for all $f\in C(X)$, where $\beta(x_1,x_2),\beta(x_3,x_4)\in\mathbb{T}^+$. Substituting $f$ by $h_{(x_1,x_2)}$, 
we deduce that $\{x_3,x_4\}=\{x_1,x_2\}$. Now since both scalars $\beta(x_1,x_2)$ and $\beta(x_3,x_4)$ are in $\mathbb{T}^+$, we get $(x_3,x_4)=(x_1,x_2)$, as desired. 

\begin{cl}\label{cl9}
For any $\{x_1,x_2\},\{x_3,x_4\}\in X_2$, we have  
$$
\left|\{x_1,x_2\}\cap\{x_3,x_4\}\right|=\left|\Lambda_Y\left( \Gamma\left( (x_1,x_2) \right) \right)\cap\Lambda_Y\left( \Gamma\left( (x_3,x_4) \right) \right)\right|.
$$
\end{cl}

Let $\{x_1,x_2\},\{x_3,x_4\}\in X_2$. If $\{x_1,x_2\}=\{x_3,x_4\}$, then either  $\Gamma\left((x_1,x_2)\right)=\Gamma\left((x_3,x_4) \right)$ or $\Gamma\left((x_1,x_2)\right)=\Gamma\left((x_4,x_3) \right)$ and thus the equality holds. 
Assume that $\{x_1,x_2\}\neq\{x_3,x_4\}$. Then $(x_1,x_2)\neq (x_3,x_4)$ and $(x_1,x_2)\neq (x_4,x_3)$. Hence 
$\Gamma\left( (x_1,x_2) \right)=(y_1,y_2)$ and $\Gamma\left( (x_3,x_4) \right)=(y_3,y_4)$ for some $\{y_1,y_2\},\{y_3,y_4\}\in Y_2$ with $\{y_1,y_2\}\neq\{y_3,y_4\}$ by the injectivity of $\Gamma$ and the fact that $\Lambda_Y \left(\Gamma\left((x_1,x_2)\right)\right)=\Lambda_Y\left(\Gamma\left( (x_2,x_1) \right) \right)$. We have two equations:
\begin{align*}
\Delta(f)(y_1)-\Delta(f)(y_2)&=\beta(x_1,x_2)\left(f(x_1)-f(x_2)\right),\\
\Delta(f)(y_3)-\Delta(f)(y_4)&=\beta(x_3,x_4)\left(f(x_3)-f(x_4)\right),
\end{align*}
for all $f\in C(X)$, where $\beta(x_1,x_2),\beta(x_3,x_4)\in\mathbb{T}^+$. Put $g=h_{(x_1,x_2)}$ and $h=h_{(x_3,x_4)}$. Then using the first equality for $g$ and the second one for $h$, we obtain
\begin{align*}
\Delta(g)(y_1)-\Delta(g)(y_2)&=\beta(x_1,x_2)\left(g(x_1)-g(x_2)\right),\\
\Delta(h)(y_3)-\Delta(h)(y_4)&=\beta(x_3,x_4)\left(h(x_3)-h(x_4)\right).
\end{align*}
By Claim \ref{cl3}, 
there exist a homeomorphism $\phi_{g,h}$ from $Y$ onto $X$, a linear functional $\mu_{g,h}$ on $C(X)$ and a number $\lambda_{g,h}\in \mathbb{T}$ with $\lambda_{g,h}\neq -\mu_{g,h}(1_X)$ such that
$$
\Delta(g)(y)=\lambda_{g,h} g(\phi_{g,h}(y))+\mu_{g,h}(g)
$$
and 
$$
\Delta(h)(y)=\lambda_{g,h} h(\phi_{g,h}(y))+\mu_{g,h}(h)
$$
for all $y\in Y$, and therefore  
\begin{align*}
\Delta(g)(y_1)-\Delta(g)(y_2)&=\lambda_{g,h}\left(g(\phi_{g,h}(y_1))-g(\phi_{g,h}(y_2))\right),\\
\Delta(h)(y_3)-\Delta(h)(y_4)&=\lambda_{g,h}\left(h(\phi_{g,h}(y_3))-h(\phi_{g,h}(y_4))\right).
\end{align*}
It follows that
\begin{align*}
\lambda_{g,h}\left(g(\phi_{g,h}(y_1))-g(\phi_{g,h}(y_2))\right)&=\beta(x_1,x_2)\left(g(x_1)-g(x_2)\right),\\
\lambda_{g,h}\left(h(\phi_{g,h}(y_3))-h(\phi_{g,h}(y_4))\right)&=\beta(x_3,x_4)\left(h(x_3)-h(x_4)\right).
\end{align*}
These equalities imply that 
$$
(\phi_{g,h}(y_1),\phi_{g,h}(y_2)))\in\left\{(x_1,x_2),(x_2,x_1)\right\}
$$
and
$$
(\phi_{g,h}(y_3),\phi_{g,h}(y_4))\in\left\{(x_3,x_4),(x_4,x_3)\right\}.
$$
Then we have four possibilities:
\begin{enumerate}
\item $x_1=\phi_{g,h}(y_1),\; x_2=\phi_{g,h}(y_2),\; x_3=\phi_{g,h}(y_3),\;x_4=\phi_{g,h}(y_4)$.
\item $x_1=\phi_{g,h}(y_1),\; x_2=\phi_{g,h}(y_2),\; x_3=\phi_{g,h}(y_4),\;x_4=\phi_{g,h}(y_3)$.
\item $x_1=\phi_{g,h}(y_2),\; x_2=\phi_{g,h}(y_1),\; x_3=\phi_{g,h}(y_3),\;x_4=\phi_{g,h}(y_4)$.
\item $x_1=\phi_{g,h}(y_2),\; x_2=\phi_{g,h}(y_1),\; x_3=\phi_{g,h}(y_4),\;x_4=\phi_{g,h}(y_3)$.
\end{enumerate}
If $\left|\{x_1,x_2\}\cap\{x_3,x_4\}\right|=1$, we infer from injectivity of $\phi_{g,h}$ that 
$$
\left|\Lambda_Y\left( \Gamma\left( (x_1,x_2) \right) \right)\cap\Lambda_Y\left( \Gamma\left( (x_3,x_4) \right) \right) \right|=\left|\{y_1,y_2\}\cap\{y_3,y_4\}\right|=1,
$$
while if $\left|\{x_1,x_2\}\cap\{x_3,x_4\}\right|=0$, then 
$$
\left|\Lambda_Y \left( \Gamma\left( (x_1,x_2) \right) \right)\cap\Gamma 
\left( (x_3,x_4) \right) \right|=\left|\{y_1,y_2\}\cap\{y_3,y_4\}\right|=0.
$$

\medskip


The proof of Claim \ref{cl10} is the same as that of Step 10 of \cite{JimSad-20}.

\begin{cl}\label{cl10}
Assume $|X|\geq 3$. For each $x\in X$ and any $\{x_1,x_2\}\in X_2$ with $x_1\neq x\neq x_2$, there exists a unique point, depending only on $x$ and denoted by $\varphi(x)$, in the intersection $\Gamma(\{x,x_1\})\cap\Gamma(\{x,x_2\})$. Then the map $\varphi\colon X\to Y$ is injective and $\{\varphi(x_1),\varphi(x_2)\}=\Lambda_Y \left( \Gamma\left( (x_1,x_2) \right) \right)$ for all $\{x_1,x_2\}\in X_2$. 
\end{cl}
Let $Y_0=\varphi(X)$. Since the map $\varphi\colon X\to Y$ is injective, its inverse $\phi_0 \colon Y_0 \to X$ is a bijection which satisfies  
$$
\{y_1,y_2\}=\Lambda_Y\left(\Gamma\left( (\phi_0(y_1),\phi_0(y_2) \right) \right)\qquad (\{y_1,y_2\}\in (Y_0)_2).
$$

Now the same argument as in Step 12 of \cite{JimSad-20} yields the next claim. 

\begin{cl}\label{cl12}
There exists a number $\lambda\in\mathbb{T}$ such that
$$
\Delta(f)(y_1)-\Delta(f)(y_2)=\lambda\left(f(\phi_0(y_1))-f(\phi_0(y_2))\right)\qquad \left(f\in C(X),\; y_1,y_2\in Y_0\right).                                            
$$
\end{cl}

Using the above claim we can define a functional $\mu\colon C(X)\to\mathbb{C}$ by 
$$
\mu(f)=\Delta(f)(y_0)-\lambda f(\phi_0(y_0)) \qquad (f\in C(X)),
$$
where $y_0$ is an arbitrary point of $Y_0$. Then it is obvious that $\mu$ is well-defined and homogeneous and, moreover,  
\begin{equation} \label{Des-Delta}
\Delta(f)(y)=\lambda f(\phi_0(y))+\mu(f) \qquad (f\in C(X),\; y\in Y_0).
\end{equation}

Note that $\Delta(1_X)$ is a nonzero constant function by Claim \ref{cl2}. Hence it follows from \eqref{Des-Delta} that $\mu(1_X)\neq -\lambda$. 

The proof of Step 15 of \cite{JimSad-20} can be applied to get the next claim.   
\begin{cl}\label{cl15}
$\phi_0\colon Y_0\to X$ is a homeomorphism.
\end{cl}
In the next claims we shall show that the homeomorphism $\phi_0\colon Y_0 \to X$ can be extended to a homeomorphism $\phi \colon Y \to X$ satisfying $\Delta(f)(y)=\lambda f(\phi(y))+\mu(f)$ for all $f\in C(X)$ and $y\in Y$. To do this we first prove the next claim.  

\begin{cl}\label{cl16}
The map $S\colon C(X)\to C(Y)$ defined by 
$$
S(f)(y)=\lambda^{-1}(\Delta(f)(y)-\mu(f))\quad (f\in C(X),\; y\in Y)
$$
is a unital algebra homomorphism.  
\end{cl} 

Fix a point $y\in Y$ and define the functional $S_y\colon C(X)\to\mathbb{C}$ by 
$$
S_y(f)=\lambda^{-1}(\Delta(f)(y)-\mu(f))\qquad (f\in C(X)).
$$
Since $\Delta(1_X)$ is a constant function it follows from the equality \eqref{Des-Delta} that $S_y(1_X)=1$. We next prove that $S_y$ is linear and multiplicative. Since $S_y(0_X)=0$, by the Kowalski--S{\l}odkowski theorem \cite{KowSlo-80} it suffices to show that $S_y(f)-S_y(g)\in (f-g)(X)$ for every $f,g\in C(X)$. Let $f,g\in C(X)$. Since $\phi_0\colon Y_0 \to X$ is a bijective map, there exists $y_0\in Y_0$ such that $\phi_0(y_0)=\phi_{f,g}(y)$. Construct the sequence $\{y_i\}_{i=0 }^\infty$ in $Y_0$ such that  
$$
\phi_0(y_{i+1})=\phi_{f,g}(y_{i})\qquad (i\in\mathbb{N}\cup\{0\}).
$$
Since $Y_0$ is a first countable compact Hausdorff space, passing through a subsequence we may assume that $\{y_i\}_i\to z_0$ for some $z_0\in Y_0$. Hence, tending $i\to\infty$ in the above equality, we get $\phi_0(z_0)=\phi_{f,g}(z_0)$. Since $z_0,y_i\in Y_0$, Claim \ref{cl12} provides the equations:
$$
\Delta(f)(z_0)-\Delta(f)(y_i)=\lambda(f(\phi_0(z_0))-f(\phi_0(y_i)))
$$
and 
$$
\Delta(g)(z_0)-\Delta(g)(y_i)=\lambda(g(\phi_0(z_0))-g(\phi_0(y_i))).
$$
On the other hand, since $\phi_{f,g}(z_0)=\phi_0(z_0)$ and $\phi_{f,g}(y_i)=\phi_0(y_{i+1})$, we have 
\begin{align*}
\Delta(f)(z_0)-\Delta(f)(y_i)&=\lambda_{f,g}(f(\phi_{f,g}(z_0))-f(\phi_{f,g}(y_i))) \\
                             &=\lambda_{f,g}(f(\phi_0(z_0))-f(\phi_0(y_{i+1})))
\end{align*}
and 
\begin{align*}
\Delta(g)(z_0)-\Delta(g)(y_i)&=\lambda_{f,g}(g(\phi_{f,g}(z_0))-g(\phi_{f,g}(y_i)))\\
                             &=\lambda_{f,g}(g(\phi_0(z_0))-g(\phi_0(y_{i+1}))).
\end{align*}
Hence, using the cited equations above, for each $i\in\mathbb{N}\cup \{0\}$ we have 
$$
f(\phi_0(z_0))-f(\phi_0(y_i))=\lambda^{-1}\lambda_{f,g}(f(\phi_0(z_0))-f(\phi_0(y_{i+1})))
$$
and 
$$
g(\phi_0(z_0))-g(\phi_0(y_i))=\lambda^{-1}\lambda_{f,g}(g(\phi_0(z_0))-g(\phi_0(y_{i+1}))).
$$
Now, it follows by induction that for each $i\in \mathbb{N}\cup \{0\}$ and $n\in \mathbb{N}$, we have 
$$
f(\phi_0(z_0))-f(\phi_0(y_i))=(\lambda^{-1}\lambda_{f,g})^n(f(\phi_0(z_0))-f(\phi_0(y_{i+n}))),
$$
and
$$
g(\phi_0(z_0))-g(\phi_0(y_i))=(\lambda^{-1}\lambda_{f,g})^n(g(\phi_0(z_0))-g(\phi_0(y_{i+n}))),
$$
Thus tending $n\to \infty$ above, we get 
$$
f(\phi_0(z_0))=f(\phi_0(y_i))\qquad (i\in \mathbb{N}\cup \{0\}). 
$$
and 
$$
g(\phi_0(z_0))=g(\phi_0(y_i))\qquad (i\in \mathbb{N}\cup \{0\}). 
$$
Therefore, for each $i\in \mathbb{N}\cup \{0\}$, we infer from the equations that  
$$
\Delta(f)(z_0)-\Delta(f)(y_i)=\lambda(f(\phi_0(z_0))-f(\phi_0(y_i)))=0,
$$
and 
$$
\Delta(g)(z_0)-\Delta(g)(y_i)=\lambda(g(\phi_0(z_0))-g(\phi_0(y_i)))=0,
$$
that is,   
$$
\Delta(f)(z_0)=\Delta(f)(y_i), \quad  \Delta(g)(z_0)=\Delta(g)(y_i) \qquad (i\in \mathbb{N}\cup \{0\}),
$$
and taking limits with $i\to\infty$ above, we deduce that  
$$
\Delta(f)(z_0)=\Delta(f)(y_0),\qquad \Delta(g)(z_0)=\Delta(g)(y_0).
$$
On the other hand, notice that $f(\phi_{f,g}(y))=f(\phi_0(y_0))=f(\phi_0(z_0))$, and consequently   
\begin{align*}
\Delta(f)(y) &=\lambda_{f,g} f(\phi_{f,g}(y))+ \mu_{f,g}(f) \\ 
        &=\lambda_{f,g} f(\phi_0(z_0))+ \mu_{f,g}(f) \\
        &=\lambda_{f,g} f(\phi_{f,g}(z_0))+ \mu_{f,g}(f)\\
        &=\Delta(f)(z_0).
\end{align*}   
Therefore we have 
$$
\Delta(f)(y)=\Delta(f)(z_0)=\Delta(f)(y_0),
$$ 
and, similarly, we can obtain 
$$
\Delta(g)(y)=\Delta(g)(z_0)=\Delta(g)(y_0).
$$
Now, using the equality \eqref{Des-Delta} and the definition of $S_y$, we can write  
\begin{align*}
\Delta(f)(y_0)&=\lambda f(\phi_0(y_0))+\mu(f)=\lambda f(\phi_0(y_0))+\Delta(f)(y)-\lambda S_y(f),\\
\Delta(g)(y_0)&=\lambda g(\phi_0(y_0))+\mu(g)=\lambda g(\phi_0(y_0))+\Delta(g)(y)-\lambda S_y(g),
\end{align*}
which imply 
\begin{align*}
S_y(f)&=f(\phi_0(y_0))+\lambda^{-1}\left(\Delta(f)(y)-\Delta(f)(y_0)\right)=f(\phi_0(y_0)),\\
S_y(g)&=g(\phi_0(y_0))+\lambda^{-1}\left(\Delta(g)(y)-\Delta(g)(y_0)\right)=g(\phi_0(y_0)).
\end{align*}
Finally, we deduce the required condition: 
$$
S_y(f)-S_y(g)=f(\phi_0(y_0))-g(\phi_0(y_0))\in(f-g)(X).
$$
Hence $S_y$ is a unital multiplicative linear functional on $C(X)$. Since $y$ was arbitrary, we conclude that $S\colon C(X)\to C(Y)$ is a unital algebra homomorphism. 

\begin{cl}\label{cl17}
There exists a homeomorphism $\phi\colon Y\to X$ such that 
$$
\Delta(f)=\lambda f\circ\phi+\mu(f)1_Y\qquad (f\in C(X)).
$$
\end{cl}

Let $S\colon C(X) \to C(Y)$ be the unital algebra homomorphism given in Claim \ref{cl16}. By Gelfand theory, $S$ induces a continuous map $\phi\colon Y\to X$ such that $S(f)=f\circ\phi$ for all $f\in C(X)$, and thus $\Delta(f)=\lambda f\circ\phi+\mu(f)1_Y$ for all $f\in C(X)$. Now, a similar proof to that of Step 17 in \cite{JimSad-20} shows that $\phi$ is a homeomorphism from $Y$ onto $X$.

We note that $\phi(y)=\phi_0(y)$ for all $y\in Y_0$, since by Claim \ref{cl17} and the equation \eqref{Des-Delta} we have $f(\phi(y))=f(\phi_0(y))$ for all $f\in C(X)$ and $y\in Y_0$. 
\begin{cl}\label{cl18}
For each $f\in C(X)$, we have $T(\pi_X(f))=\pi_Y(\lambda f\circ\phi)$. In particular, $T$ is linear and surjective. 
\end{cl} 

Let $f\in C(X)$. By Claim \ref{cl17} and the definition of $\Delta$, we have 
$$
\lambda f\circ\phi+\mu(f)1_Y=\Delta(f)=\Psi_Y^{-1}(T(\pi_X(f)),f(u_0)).
$$
Hence $\Psi_Y(\lambda f\circ\phi+\mu(f)1_Y)=(T(\pi_X(f)),f(u_0))$ which implies  
$$
\pi_Y(\lambda f\circ\phi)=\pi_Y(\lambda f\circ\phi+\mu(f)1_Y)=T(\pi_X(f)).
$$ 
This completes the proof of Theorem \ref{main2}. 
\end{proof}


\textbf{Acknowledgements.} Research partially supported by Junta de Andaluc\'{\i}a grant FQM194.

\end{document}